\documentclass[11pt]{amsart}
\usepackage[margin=1in]{geometry}
\usepackage{latexsym}
\usepackage{amsfonts}
\usepackage{amsmath}
\usepackage{amssymb}
\usepackage{amsthm}
\usepackage{enumerate}
\usepackage[hang,flushmargin]{footmisc}
\usepackage{caption}
\usepackage{tabu}
\usepackage{mathrsfs}
\usepackage{amsaddr}

\usepackage{graphicx}
\usepackage{epstopdf}
\usepackage{epsfig}
\usepackage{caption}
 
\usepackage{bm} 
\usepackage{cite}

\newtheorem{theorem}{Theorem}

\theoremstyle{definition}

\begin{document}
\title{Enumerating Cliques in Direct Product Graphs}
\author{Colin Defant}
\address{Princeton University \\ Fine Hall, 304 Washington Rd. \\ Princeton, NJ 08544}
\email{cdefant@princeton.edu}

\begin{abstract}
The unitary Cayley graph of $\mathbb Z/n\mathbb Z$, denoted $G_{\mathbb Z/n\mathbb Z}$, is the graph with vertices $0,1,\ldots,$ $n-1$ in which two vertices are adjacent if and only if their difference is relatively prime to $n$. These graphs are central to the study of graph representations modulo integers, which were originally introduced by Erd\H{o}s and Evans. We give a brief account of some results concerning these beautiful graphs and provide a short proof of a simple formula for the number of cliques of any order $m$ in the unitary Cayley graph $G_{\mathbb Z/n\mathbb Z}$. This formula involves an exciting class of arithmetic functions known as Schemmel totient functions, which we also briefly discuss. More generally, the proof yields a formula for the number of cliques of order $m$ in a direct product of balanced complete multipartite graphs.   
\end{abstract}

\maketitle

\bigskip



\section{Unitary Cayley Graphs and Schemmel Totient Functions}
Let $R$ be a commutative ring with unity. The \emph{unitary Cayley graph} of $R$, denoted $G_R$, is the graph whose vertices are the elements of $R$ in which two vertices are adjacent if and only if their difference is a unit in $R$. In symbols, $G_R$ has vertex set $V(G_R)=R$ and edge set $E(G_R)=\{\{x,y\}\colon x-y\in R^\times\}$. Unitary Cayley graphs have featured prominently in the literature of the past two decades \cite{Akhtar, Beaudrap, Berrizbeitia, burcroff, Defant4, Dejter, defantandiyer, Fuchs, Ilic, Kiani, Kiani2, Klotz, Liu, Liu2, Madhavi, Maheswari, Maheswari2, Ramaswamy}, most commonly in the special case $R=\mathbb Z/n\mathbb Z$ for some integer $n\geq 2$. 

One can view the unitary Cayley graph $G_{\mathbb Z/n\mathbb Z}$ as the graph with vertices $0,1,\ldots,n-1$ in which two vertices are adjacent if and only if their difference is relatively prime to $n$. This number-theoretic definition leads to several interesting number-theoretic properties of these graphs. For example, the number of edges in $G_{\mathbb Z/n\mathbb Z}$ is given by \begin{equation}\label{Eq3}
|E(G_{\mathbb Z/n\mathbb Z})|=\frac{1}{2}n\varphi(n),
\end{equation} where $\varphi$ is Euler's totient function. The clique number of $G_{\mathbb Z/n\mathbb Z}$, defined to be the largest integer $k$ such that $G_{\mathbb Z/n\mathbb Z}$ has a clique of order $k$, turns out to be the smallest prime factor of $n$. This, in turn, is also the chromatic number of $G_{\mathbb Z/n\mathbb Z}$ \cite{Klotz}. Klotz and Sander have also shown that the eigenvalues of $G_{\mathbb Z/n\mathbb Z}$ (that is, the eigenvalues of an adjacency matrix of $G_{\mathbb Z/n\mathbb Z}$) are integers that divide $\varphi(n)$ (in fact, they are given by Ramanujan sums) \cite{Klotz}. Figure \ref{Fig1} depicts the graphs $G_{\mathbb Z/n\mathbb Z}$ for $2\leq n\leq 10$.   

\begin{figure}[t]
\begin{center}\includegraphics[height=6.5cm]{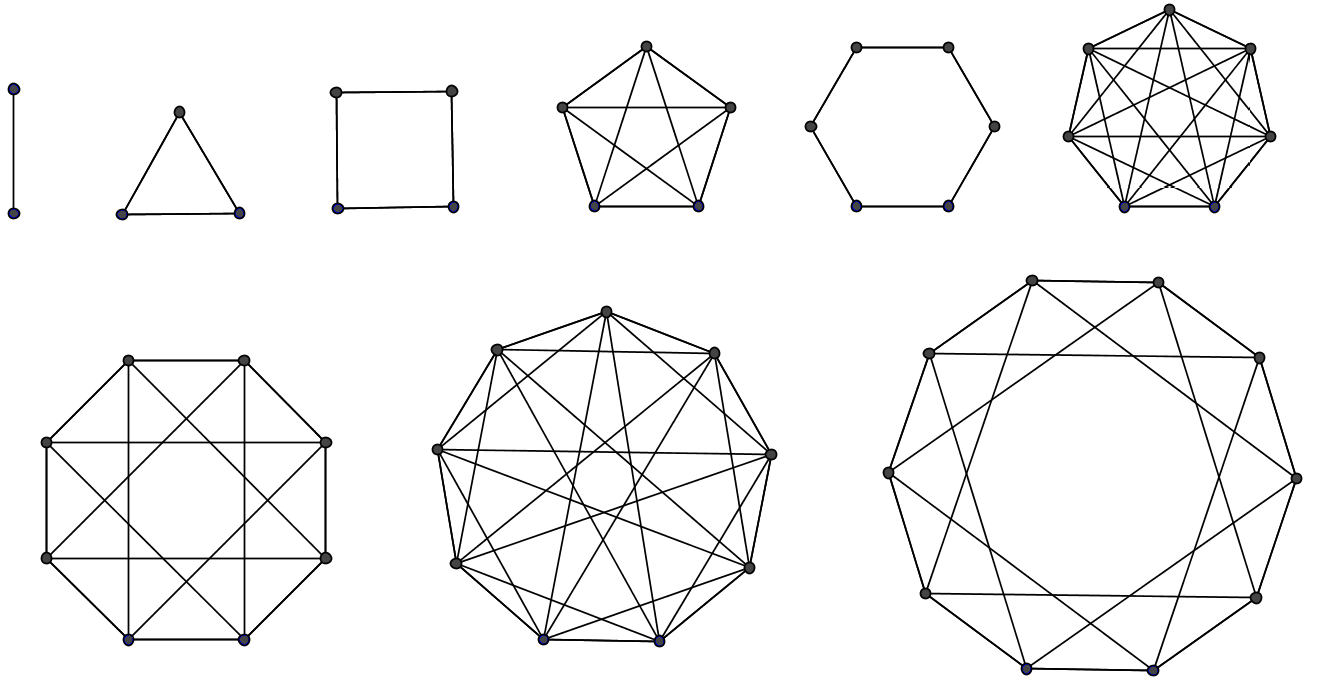}
\end{center}
\captionof{figure}{The unitary Cayley graphs of $\mathbb Z/n\mathbb Z$ for $2\leq n\leq 10$.} \label{Fig1}
\end{figure}

One motivation for studying the unitary Cayley graphs $G_{\mathbb Z/n\mathbb Z}$, other than their inherent beauty and interesting properties, arises from the theory of graph representations modulo integers. Erd\H{o}s and Evans \cite{Erdos} defined a graph $G$ to be \emph{representable modulo $n$} if there exists a labeling of the vertices of $G$ with distinct elements of $\{1,2,\ldots,n\}$ such that two vertices are adjacent if and only if the difference between their labels is relatively prime to $n$. In other words, $G$ is representable modulo $n$ if it is isomorphic to an induced subgraph of $G_{\mathbb Z/n\mathbb Z}$. These authors then proved that every finite simple graph is representable modulo some positive integer. The \emph{representation number} of a graph $G$ is the smallest integer $n$ such that $G$ is representable modulo $n$. This definition has garnered a huge amount of interest as researchers have investigated the representation numbers of various graphs \cite{Akhtar, Akhtar2, Akhtar3, Evans, Evans2, Narayan}. See Gallian's ``Dynamic Survey of Graph Labeling" for more information about the representation numbers of graphs and for additional references \cite{Gallian}. The topic of graph representations modulo integers, which revolves around induced subgraphs of the graphs $G_{\mathbb Z/n\mathbb Z}$, certainly motivates the study of such graphs. 

When Dejter and Giudici introduced unitary Cayley graphs in 1995, they showed that
\begin{equation}\label{Eq1}
T(G_{\mathbb Z/n\mathbb Z})=\frac 16n\varphi(n)S_2(n),
\end{equation} 
where $T(G)$ denotes the number of triangles in the graph $G$ and $S_2$ is the $2^\text{nd}$ Schemmel totient function \cite{Dejter}. For each nonnegative integer $r$, the $r^\text{th}$ Schemmel totient function $S_r$ is the multiplicative arithmetic function that satisfies 
\begin{equation} \label{Eq2} 
S_r(p^{\alpha})=\begin{cases} p^{\alpha-1}(p-r), & \mbox{if } p\geq r; \\ 0, & \mbox{if } p<r \end{cases} 
\end{equation} 
for all primes $p$ and positive integers $\alpha$ (here, ``multiplicative" means that $S_r(ab)=S_r(a)S_r(b)$ whenever $\gcd(a,b)=1$). Note that $S_0(n)=n$. 

As their name suggests, the Schemmel totient functions are generalizations of Euler's totient function that were originally introduced by Schemmel \cite{Schemmel}. Indeed, $S_1=\varphi$. The standard combinatorial interpretation of the Euler totient function is that $\varphi(n)$ is the number of integers less than or equal to $n$ that are positive and relatively prime to $n$. Note that $S_0(n)=n$ has a similar combinatorial interpretation: it is the number of integers less than or equal to $n$ that are positive! More generally, $S_r(n)$ is the number of positive integers $k\leq n$ such that $\gcd(k+i,n)=1$ for all $i\in\{0,1,\ldots,r-1\}$ \cite{Schemmel}. The Schemmel totient functions seldom appear in the wild, which is unfortunate because they have a certain attractive mystique. It is worth noting, however, that Lehmer found the Schemmel totient functions emerge in the solutions of certain enumerative problems concerning magic squares \cite{Lehmer}. Moreover, the number-theoretic properties of the Schemmel totient functions have been studied in their own right \cite{Defant, Defant2, Defant3, McCarthy, Morgado, Subbarao, Subbarao2}. 

In 2007, Klotz and Sander \cite{Klotz} gave an alternative proof of \eqref{Eq1}. Madhavi and Maheswari \cite{Madhavi} then (apparently independently) rediscovered this result in 2010. These two papers and the original paper of Dejter and Giudici are certainly interesting, but they all fail to phrase the formula in \eqref{Eq1} in the ``correct" way. These papers all prove that the number of triangles in $G_{\mathbb Z/n\mathbb Z}$ is $\dfrac 16n\varphi(n)S_2(n)$. The ``correct" way to phrase this theorem is as follows: 

\begin{center} 
The number of cliques of order $3$ in $G_{\mathbb Z/n\mathbb Z}$ is
\end{center} \[\frac{S_0(n)}{1}\cdot\frac{S_1(n)}{2}\cdot\frac{S_2(n)}{3}.\]   

The number of cliques of order $1$ (i.e., the number of vertices) in $G_{\mathbb Z/n\mathbb Z}$ is simply $n$, which we can write as $\dfrac{S_0(n)}{1}$. According to \eqref{Eq3}, the number of cliques of order $2$ (which is simply the number of edges) in $G_{\mathbb Z/n\mathbb Z}$ is $\dfrac 12n\varphi(n)$, which we may rewrite as $\dfrac{S_0(n)}{1}\cdot\dfrac{S_1(n)}{2}$. This naturally leads us to speculate that the number of cliques of order $m$ in $G_{\mathbb Z/n\mathbb Z}$ is \begin{equation}\label{Eq5} 
\prod_{k=1}^m\frac{S_{k-1}(n)}{k}.
\end{equation} This assertion does indeed hold, as the author proved in slightly greater generality in \cite{Defant4}. We mentioned before that the clique number of $G_{\mathbb Z/n\mathbb Z}$ is equal to the smallest prime factor of $n$; note that this follows as an easy corollary to the above formula \eqref{Eq5}. 

The primary purpose of this article is to give a simplified proof of the formula \eqref{Eq5} in a much more natural and general framework. Namely, we will prove a formula for the number of cliques of order $m$ in a direct product of balanced complete multipartite graphs. 

\section{Direct Products of Balanced complete Multipartite Graphs}

Let $G$ be a graph with vertex set $V(G)$. We say $G$ is a \emph{complete $b$-partite graph} if there is a partition of $V(G)$ into $a$ parts $B_1,\ldots,B_b$, called the \emph{partite sets}, such that two vertices are adjacent if and only if they do not belong to the same partite set. A complete multipartite graph is called \emph{balanced} if the partite sets all have the same cardinality. Let $K[a,b]$ denote the balanced complete $b$-partite graph in which every partite set contains $a$ vertices. For any prime $p$ and positive integer $\alpha$, $G_{\mathbb Z/p^\alpha\mathbb Z}\cong K[p^{\alpha-1},p]$. Indeed, viewing the vertices of $G_{\mathbb Z/p^\alpha\mathbb Z}$ as $0,1,\ldots,p^\alpha-1$, the partite sets are simply the different residue classes modulo $p$. 

The \emph{direct product} (also called the \emph{tensor product}, \emph{Kronecker product}, \emph{weak product}, or \emph{conjunction}) of graphs $H_1,\ldots,H_r$ with vertex sets $V(H_1),$ $\ldots,V(H_r)$, denoted $\prod_{i=1}^r H_i$, is a graph whose vertex set is the cartesian product $V(H_1)\times\cdots\times V(H_r)$. Two vertices $(y_1,\ldots,y_r)$ and $(z_1,\ldots,z_r)$ of the direct product are adjacent if and only if $y_i$ is adjacent to $z_i$ in $H_i$ for all $1\leq i\leq r$. It follows immediately from the Chinese remainder theorem and the preceding paragraph that $G_{\mathbb Z/n\mathbb Z}$ is isomorphic to a direct product of balanced complete multipartite graphs. More specifically, if $p_1^{\alpha_1}\cdots p_r^{\alpha_r}$ is the prime factorization of $n$, then $G_{\mathbb Z/n\mathbb Z}\cong\prod_{i=1}^r G_{\mathbb Z/p_i^{\alpha_i}\mathbb Z}\cong\prod_{i=1}^r K[p_i^{\alpha_i-1},p_i]$. 

In fact, if $R$ is any finite commutative ring with unity, then the authors of \cite{Akhtar4} have shown that the unitary Cayley graph $G_R$ is isomorphic to a direct product of balanced complete multipartite graphs. Their argument boils down to observing that every finite ring is a direct product (as a ring) of finite local rings and then showing that the unitary Cayley graph of a finite local ring is isomorphic to a balanced complete multipartite graph. This suggests that it is natural to study direct products of balanced complete multipartite graphs as generalizations of (finite) unitary Cayley graphs. 

For any positive integers $x,y,m$, put $\mathcal S_m(x,y)=\max\{x(y-m),0\}$. If $p$ is a prime and $\alpha$ is a positive integer, then $S_m(p^\alpha)=\mathcal S_m(p^{\alpha-1},p)$, where $S_m$ denotes the $m^\text{th}$ Schemmel totient function. We are now in a position to state and prove our generalization of the formula in \eqref{Eq5}.

\begin{theorem}
Let $a_1,\ldots,a_r,b_1,\ldots,b_r,m$ be positive integers. Let $K[a_i,b_i]$ be the balanced complete $b_i$-partite graph in which each partite set contains $a_i$ vertices. Let $X=\prod_{i=1}^rK[a_i,b_i]$ be the direct product of the graphs $K[a_i,b_i]$. The number of cliques of order $m$ in $X$ is \[\frac{1}{m!}\prod_{k=1}^m\prod_{i=1}^r\mathcal S_{k-1}(a_i,b_i).\]
\end{theorem}
\begin{proof}
When $m=1$, the theorem states that $X$ has $\prod_{i=1}^ra_ib_i$ vertices, which is certainly true. We proceed by induction on $m$. Let $\text{CL}(t)$ denote the set of cliques of order $t$ in $X$. We simply need to show that 
\begin{equation}\label{Eq4} 
|\text{CL}(m+1)|=\dfrac{|\text{CL}(m)|}{m+1}\prod_{i=1}^r\mathcal S_m(a_i,b_i).
\end{equation} 
This is obvious if $\text{CL}(m)=\emptyset$, so assume $\text{CL}(m)$ is nonempty. Of course, $(m+1)|\text{CL}(m+1)|$ is the number of pairs $(w,C)$, where $C\in \text{CL}(m+1)$ and $w\in C$. This is also the number of pairs $(w,D)$, where $D\in\text{CL}(m)$ and $w$ is a vertex of $X$ that is adjacent to every element of $D$ (if $w$ is such a vertex, then $w\not\in D$ since no vertex is adjacent to itself). Thus, it suffices to show that for every $D\in\text{CL}(m)$, there are precisely $\prod_{i=1}^r\mathcal S_m(a_i,b_i)$ vertices that are adjacent to every element of $D$.  

Choose $D=\{x_1,\ldots,x_m\}\in\text{CL}(m)$. Recall that each vertex $x_j$ is a vertex in a direct product graph. Thus, we may write $x_j=(y_{1j},\ldots,y_{rj})$, where each $y_{ij}$ is a vertex in $K[a_i,b_i]$. The vertices of $X$ adjacent to every vertex in $D$ are precisely the tuples $(z_1,\ldots,z_r)$ such that $z_i$ is adjacent to $y_{ij}$ for all $i\in\{1,2,\ldots,r\}$ and $j\in\{1,2,\ldots,m\}$. It follows from the fact that $D$ is a clique that for each $i\in\{1,2,\ldots,r\}$, the vertices $y_{i1},\ldots,y_{im}$ are in distinct partite sets in $K[a_i,b_i]$. This implies that there are $a_i(b_i-m)=\mathcal S_m(a_i,b_i)$ choices for $z_i$. Hence, the total number of vertices $(z_1,\ldots,z_r)$ that are adjacent to all vertices in $D$ is $\prod_{i=1}^r\mathcal S_m(a_i,b_i)$. \qedhere   
\end{proof}

\end{document}